\documentclass[10pt,a4paper]{amsart}
\usepackage{amssymb,amsmath,amsfonts}

\headsep=1.2500cm \numberwithin{equation}{section}
\newtheorem{Theorem}{Theorem}[section]
\newtheorem{Proposition}[Theorem]{Proposition}
\newtheorem{cor}[Theorem]{Corollary}
\newtheorem{lemma}[Theorem]{Lemma}
\theoremstyle{remark}

\newtheorem{Example}[Theorem]{Example}

\newtheorem{Remark}[Theorem]{Remark}

\begin{document}
\title{Generalized notions of amenability for a class of matrix algebras}

\author{A. Sahami}
\address{Faculty of Basic sciences, Department of Mathematics, Ilam University, P.O.Box 69315-516, Ilam,
Iran.} \email{amir.sahami@aut.ac.ir}


\begin{abstract}
We investigate the notions of amenability and its related
homological notions for a class of $I\times I$-upper triangular
matrix algebra, say $UP(I,A)$, where $A$ is a Banach algebra
equipped with a non-zero character. We show that $UP(I,A)$ is
pseudo-contractible (amenable) if and only if $I$ is singleton and
$A$ is pseudo-contractible (amenable), respectively. We also study
the notions of pseudo-amenability and approximate biprojectivity of
$UP(I,A)$.
\end{abstract}



\maketitle

\section{Introduction and Preliminaries}
B. E. Johnson studied  the class  of amenable Banach algebras.
Indeed a Banach algebra $A$ is amenable if every continuous
derivation $D:A\rightarrow X^{*}$ is inner, for every Banach
$A$-bimodule $X$, that is, there exists $x_{0}\in X^{*}$ such that
$$D(a)=a\cdot x_{0}-x_{0}\cdot a\quad(a\in A).$$ He also showed that  $A$ is amenable if and only if
there exists a bounded net $(m_{\alpha})$ in $A\otimes_{p}A$ such
that $$a\cdot m_{\alpha}-m_{\alpha}\cdot a\rightarrow 0,\quad
\pi_{A}(m_{\alpha})a\rightarrow a\qquad(a\in A),$$ where
$\pi_{A}:A\otimes_{p}A\rightarrow A$ is given by $\pi_{A}(a\otimes
b)=ab$ for every $a,b\in A$, see \cite{Joh}. About the same time A.
Ya. Helemskii defined the homological notions of biflatness and
biprojectivity for Banach algebras. In fact a Banach algebra $A$ is
called biflat (biprojective), if there exists a bounded $A$-bimodule
morphism $\rho:A\rightarrow (A\otimes_{p}A)^{**}$
($\rho:A\rightarrow A\otimes_{p}A$) such that
$\pi_{A}^{**}\circ\rho$ is the canonical embedding of $A$ into
$A^{**}$ ($\rho$ is a right inverse for $\pi_{A}$), respectively see
\cite{hel}. Note that a Banach algebra $A$ is amenable if and only
if $A$ is biflat and $A$ has a bounded approximate identity. It is
known that for a locally compact group $G$, $L^{1}(G)$ is biflat
(biprojective) if and only if $G$ is amenable(compact),
respectively. Amenability of some matrix algebras studied by
Esslamzadeh \cite{Ess} and also biflatness and biprojectivity of
some semigroup algebras related to matrix algebras investigated by
Ramsden in \cite{rams}.

Recently approximate versions of amenability and homological
properties of Banach algebras have been under more observations. In
\cite{zhang} Zhang  introduced the notion of approximately
biprojective Banach algebras, that is, $A$ is approximately
biprojective if there exists a net of $A$-bimodule morphism
$\rho_{\alpha}:A\rightarrow A\otimes_{p}A$ such that
$$\pi_{A}\circ\rho_{\alpha}(a)\rightarrow a\quad(a\in A).$$ Author
with A. Pourabbas investigated approximate biprojectivity of
$2\times 2$ upper triangular Banach algebra which is a matrix
algebra, also we characterized approximate biprojectivity of Segal
algebras and weighted group algebras. We show that a Segal algebra
$S(G)$ is approximately biprojective if and only if $G$ is compact
and also we show that $L^{1}(G,w)$ is approximately biprojective if
and only if $G$ is compact, provided that $w\geq 1$ is a continuous
weight function, see \cite{sah col} and \cite{sah3}.

 Approximate amenable Banach algebras have been introduced by
Ghahramani and Loy. Indeed a Banach algebra $A$ is approximate
amenable if for every continuous derivation $D:A\rightarrow X^{*}$,
there exists a net $(x_{\alpha})$ in $X^{*}$ such that
$$D(a)=\lim_{\alpha}a\cdot x_{\alpha}-x_{\alpha}\cdot a \quad(a\in
A).$$ Other extensions of amenability are pseudo-amenability and
pseudo-contractibility. A Banach algebra $A$ is pseudo-amenable
(pseudo-contractible) if there exists a not necessarily bounded net
$(m_{\alpha})$ in $A\otimes_{p}A$ such that
$$a\cdot m_{\alpha}-m_{\alpha}\cdot a\rightarrow 0,\quad(a\cdot m_{\alpha}=m_{\alpha}\cdot a),\qquad \pi_{A}(m_{\alpha})a\rightarrow a\quad (a\in A).$$
For more information about these new concepts the reader referred to
\cite{ghah pse}, \cite{gen 1} and \cite{gen 2}. Recently in
\cite{essmail arch} and \cite{essmail sem} pseudo-amenability and
pseudo-contractibility of certain semigroup algebras,  using the
properties of matrix algebras, have been studied.

In this paper, we investigate amenability and its related
homological notions for a class of matrix algebras. We show that for
a Banach algebra $A$ with a non-zero character,  $I\times I$ upper
triangular Banach algebra $UP(I,A)$ is pseudo-contractible
(amenable) if and only if $I$ is singleton and $A$ is
pseudo-contractible (amenable), respectively. Also we characterize
whether $UP(I,A)$ is approximate amenable, pseudo-amenable and
approximate biprojective. The paper concluded by studying
amenability and approximate biprojectivity of some semigroup
algebras related to a matrix algebra.

We remark some standard notations and definitions that we shall need
in this paper. Let $A$ be a Banach algebra. Throughout this paper
the character space of $A$ is denoted by $\Delta(A)$, that is, all
non-zero multiplicative linear functionals on $A$.  Let $A$  be a
Banach algebra.  The projective tensor product
 $A\otimes_{p}A$ is a Banach $A$-bimodule via the following actions
$$a\cdot(b\otimes c)=ab\otimes c,~~~(b\otimes c)\cdot a=b\otimes
ca\hspace{.5cm}(a, b, c\in A).$$

Let $A$ be a Banach algebra and $I$ be a non-empty set. $UP(I,A)$ is
denoted for  the set of all $I\times I$ upper triangular matrices
which entries come from $A$ and
$$||(a_{i,j})_{i,j\in I}||=\sum_{i,j\in I}||a_{i,j}||<\infty.$$ With
the usual matrix operations and $||\cdot||$ as a norm, $UP(I,A)$
becomes a Banach algebra.
\section{a class of matrix algebras and generalized notions of amenability}

In this section we investigate generalized notions of amenability
for upper triangular Banach algebras.

We remind that a Banach algebra $A$ with $\phi\in\Delta(A)$ is
called left(right) $\phi$-contractible, if there exists $m\in A$
such that $am=\phi(a)m(ma=\phi(a)m)$ and $\phi(m)=1$ for every $a\in
A$, respectively. For more information the reader referred to
\cite{nas}.
\begin{Theorem}\label{main}
Let $I$ be a non-empty set and $A$ be a unital Banach algebra with
$\Delta(A)\neq \emptyset.$ $UP(I,A)$ is pseudo-contractible if and
only if $I$ is singleton and $A$ is pseudo-contractible.
\end{Theorem}
\begin{proof}
Let $UP(I,A)$ be pseudo-contractible. Then $UP(I,A)$ has a central
approximate identity, say $(e_{\alpha})$. Put $F_{i,j}$ for a matrix
 belongs to $UP(I,A)$ which $(i,j)$-th entry is $e_{A}$ and others
 are zero, where  $e_{A}$ is  an identity of  $A$. Thus $F_{i,j}e_{\alpha}=e_{\alpha}F_{i,j}$ for every $i,j\in
 I.$ This equation implies that the entries on main diagonal of
 $e_{\alpha}$ is
 equal.  Suppose conversely that $I$ is infinite. Since the entries on main diagonal of $e_{\alpha}$ are
 equal, it implies that $||e_{\alpha}||=\infty$ or the main diagonal of $e_{\alpha}$ is zero. In the case $||e_{\alpha}||=\infty$,  $e_{\alpha}$
 does not belong to $UP(I,A)$ which is impossible.  Otherwise if the main diagonal of $e_{\alpha}$ is
 zero, then $e_{\alpha}F_{i,i}=0$. Thus $0=e_{\alpha}F_{i,i}\rightarrow
 F_{i,i}$ which is impossible, hence $I$ must be finite. Suppose
 that $I=\{i_{1},i_{2},...,i_{n}\}$ and $\phi\in\Delta(A)$. Define
 $\psi\in\Delta(UP(I,A))$ by $\psi((a_{i,j})_{i,j\in I})=\phi(a_{i_{n},i_{n}})$ for every
 $(a_{i,j})\in UP(I,A)$. Since $UP(I,A)$ is pseudo-contractible, by
 \cite[Theorem 1.1]{alagh1}
  $UP(I,A)$ is left and right $\psi$-contractible. Set $$J=\{(a_{i,j})\in
 UP(I,A)|a_{i,j}=0,\text{for all}\quad j\neq i_{n}\}.$$ It is clear $J$ is a closed ideal of $UP(I,A)$ and  $\psi|_{J}\neq 0$, hence by
 \cite[Proposition 3.8]{nas}
 $J$ is left and right  $\psi$-contractible.
So there exist $m_{1},m_{2}\in J$ such that $jm_{1}=\psi(j)m_{1}$
and $m_{2}j=\psi(j)m_{2}$ and also $\psi(m_{1})=\psi(m_{2})=1$ for
each $j\in J.$ Set $m=m_{1}m_{2}\in J.$ Clearly  we have
\begin{equation}\label{eq main}
jm=mj=\psi(j)m,\quad
\psi(m)=\psi(m_{1}m_{2})=\psi(m_{1})\psi(m_{2})=1,\qquad (j\in J).
\end{equation}
Suppose conversely that $|I|>1.$ Set $m$ for the matrix with n-th
columns $(x_{1},x_{2},...,x_{n})^{t}$, where $x_{i}\in A$ for each
$i\in\{1,2,...,n\}$. Let $a$ be an element of $J$ which its n-th
columns has the form $(0,0,...,a_{n})^{t}$ for an arbitrary element
$a_{n}\in A$. Applying (\ref{eq main}) we have
$$x_{1}a_{n}=x_{2}a_{n}=...=x_{n-1}a_{n}=0,\quad
\phi(a_{n})x_{1}=\phi(a_{n})x_{2}=...=\phi(a_{n})x_{n-1}=0,$$ and
also $$a_{n}x_{n}=x_{n}a_{n}=\phi(a_{n})x_{n},\quad \phi(x_{n})=1.$$
Pick an element $a_{n}\in A$ such that $\phi(a_{n})=1.$ Applying
(\ref{eq main}) follows that $x_{1}=x_{2}=...=x_{n-1}=0$. Then $m$
becomes a matrix which its n-th columns has the form
$(0,0,...,0,x_{n})^{t}$. Set $b$ for a matrix in $J$ which its n-th
columns has the form $(b_{1},b_{2},...,b_{n-1},b_{n})^{t}$, where
$b_{n}\in \ker\phi$ and
$\phi(b_{1})=\phi(b_{2})=...=\phi(b_{n-1})=1.$ Applying (\ref{eq
main}) we have $a_{1}x_{n}=0$. Taking $\phi$ on this equation we
have $0=\phi(a_{1}x_{n})=\phi(a_{1})\phi(x_{n})=1$ which is a
contradiction. Therefore $I$ must be singleton. So $A$ is
pseudo-contractible.

 Converse is clear.
\end{proof}
Suppose that $A$ is a Banach algebra and $\phi\in\Delta(A)$. $A$ is
called (approximately) left $\phi$-amenable if there exists (a not
necessarily) bounded net $(m_{\alpha})$ in $A$ such that
$$am_{\alpha}-\phi(a)m_{\alpha}\rightarrow 0\quad \phi(m_{\alpha})\rightarrow1\qquad (a\in A),$$
respectively. Right cases define similarly. For more information
about these new concepts of amenability and its related homological
notions  see \cite{agha}, \cite{kan}, \cite{Hu} and \cite{sah1}.

\begin{Theorem}\label{main1}
Let $I$ be an  ordered set with an smallest element. Also let $A$ be
a  Banach algebra with a left unit such that $\Delta(A)\neq
\emptyset.$  $UP(I,A)$ is  pseudo-amenable (approximate amenable) if
and only if $I$ is singleton and $A$ is pseudo-amenable(approximate
amenable), respectively.
\end{Theorem}
\begin{proof}
Here we proof the pseudo-amenable case, approximate amenability is
similar. Suppose that $UP(I,A)$ is pseudo-amenable. Then there
exists a net $(m_{\alpha})$ in $UP(I,A)\otimes_{p}UP(I,A)$ such that
$$a\cdot m_{\alpha}-m_{\alpha}\cdot a\rightarrow 0,\quad \pi_{UP(I,A)}(m_{\alpha})a\rightarrow a\qquad(a\in
UP(I,A)).$$ Let $i_{0}$ be a smallest element of $I$. It is easy to
see that $\psi$ given by $\psi(a)=\phi(a_{i_{0},i_{0}})$ is a
character on $UP(I,A),$ for each $a=(a_{i,j})\in UP(I,A)$. Define
$$T:UP(I,A)\otimes_{p}UP(I,A)\rightarrow UP(I,A)$$ by $T(a\otimes
b)=\psi(a)b$ for each $a,b\in UP(I,A)$. It is easy to see that $T$
is a bounded linear map which satisfies the following:
$$T(a\cdot x)=\psi(a)T(x),\quad T(x\cdot a)=T(x)a,\quad \psi\circ
T(x)=\psi\circ \pi_{UP(I,A)}(x),$$ for each $a,b\in UP(I,A)$ and $
x\in UP(I,A)\otimes_{p}UP(I,A)$. Thus we have
$$\psi(a)T(m_{\alpha})-T(m_{\alpha})a=T(a\cdot m_{\alpha}-m_{\alpha}\cdot a)\rightarrow 0$$
and $\psi\circ T(m_{\alpha})=\psi\circ
\pi_{UP(I,A)}(m_{\alpha})\rightarrow 1$. Hence $UP(I,A)$ is
approximately right $\psi$-amenable. Using the same arguments as in
the proof of Theorem \ref{main} and applying \cite[Proposition
5.1]{saha} one can see that $I$ is singleton and $A$ is
pseudo-amenable.

Converse is  clear.
\end{proof}

Let $A$ be a Banach algebra and $a\in A.$ By $a\varepsilon_{i,j}$ we
mean a matrix belongs to $UP(I,A)$ with $(i,j)$-th place is $a$ and
zero elsewhere.
\begin{Theorem}
Let $I$ be non-empty set and  $A$ be a  Banach algebra  such that
$\Delta(A)\neq \emptyset.$  $UP(I,A)$ is amenable if and only if $I$
is singleton and $A$ amenable.
\end{Theorem}
\begin{proof}
Let $UP(I,A)$ be amenable. Then $UP(I,A)$  has a bounded approximate
identity, say $(E^{\alpha})$. Let $M>0$ be a bound for
$(E^{\alpha})$. We claim that $A$ has a bounded left approximate
identity. To see this, fix $k,l\in I.$ Then for each $a\in A$, we
have
\begin{equation}
\begin{split}
0=\lim_{\alpha}||E^{\alpha}a\varepsilon_{k,l}-a\varepsilon_{k,l}||&=\lim_{\alpha}||(\sum_{i,j}
E^{\alpha}_{i,j}\varepsilon_{i,j})a\varepsilon_{k,l}-a\varepsilon_{k,l}||\\
&=\lim_{\alpha}||\sum_{i}
E^{\alpha}_{i,l}a\varepsilon_{i,l}-a\varepsilon_{k,l}||\\
&=\lim_{\alpha}(||\sum_{i\neq k}
E^{\alpha}_{i,l}a||+||E^{\alpha}_{k,l}a-a||.
\end{split}
\end{equation}
Thus $e_{\alpha}=E^{\alpha}_{k,l}$ is a left approximate identity of
$A.$ It is easy to see that $||e_{\alpha}||\leq ||E^{\alpha}||\leq
M$. So $(e_{\alpha})$ is a bounded left approximate identity for
$A.$ We claim that $I$ is finite. Suppose conversely that $I$ is
infinite.  Pick $a\in A$ such that $||a||=1.$ Since $(e_{\alpha})$
is a bounded left approximate identity for $A$, then
$\lim_{\alpha}e_{\alpha}a=a$, for each $a\in A.$ Thus there exists a
$\alpha_{l,k}$ such that $\alpha\geq \alpha_{k,l}$ such that
$\frac{1}{2}<||e_{\alpha}a||$. Hence for $\alpha\geq \alpha_{k,l}$
we have
\begin{equation}\label{eq}
\frac{1}{2}<||e_{\alpha}a||\leq ||e_{\alpha}||=||E_{k,l}^{\alpha}||.
\end{equation}
 Since $I$ is infinite we can choose $N\in
\mathbb{N}$ such that $N>2M.$ Then choose distinct $k_{1}, l_{1},
k_{2},l_{2},...,k_{N}, l_{N}$ in $I$ and $\alpha\geq
\alpha_{k_{i},l_{i}},$ $i=1,2,...,N$. Using (\ref{eq}) one can see
that
$$M<\frac{1}{2}N= \sum^{N}_{i=1}||E_{k_i,l_{i}}^{\alpha}||\leq \sum_{i,j\in I}||E_{i,j}^{\alpha}||\leq M,$$
which is a contradiction. So $I$ is finite.

Applying the same method as in the proof of previous Theorem, it is
easy to see
 that $I$ must be singleton, then $A$ is amenable.
\end{proof}
\section{a class of Matrix algebra and approximate biprojectivity}

In this section we study approximate biprojectivity of some matrix
algebra. We also investigate the relation of approximate
biprojectivity and discreteness of maximal ideal space of a Banach
algebra.
\begin{Theorem}
Let $I$ be an  ordered set with an smallest element. Also let $A$ be
a  Banach algebra with a right identity such that $\Delta(A)\neq
\emptyset.$ $UP(I,A)$ is  approximately biprojective if and only if
$I$ is singleton and $A$ is approximately biprojective.
\end{Theorem}
\begin{proof}
Let $i_{0}$ be smallest element of $I$. Define
$\psi\in\Delta(UP(I,A))$ by $\psi(a)=\phi(a_{i_{0},i_{0}})$, where
$a=(a_{i,j})\in UP(I,A)$. Suppose that $UP(I,A)$ is  approximately
biprojective. Since $A$ has a right identity, by \cite[Lemma
5.2]{saha}, $UP(I,A)$ has a right approximate identity. Applying
\cite[Theorem 3.9]{sah3}, $UP(I,A)$ is right $\psi$-contractible.
Using the same arguments as in the proof of the  Theorem \ref{main},
$I$ is singleton and $A$ is approximately biprojective.

Converse is clear.
\end{proof}
\begin{Remark}
Let $A$ be a Banach algebra with a left approximate identity and $I$
be a finite set which has at least two elements. Then  $UP(I,A)$  is
never approximately biprojective. To see this, since
$I=\{i_{1},i_{2},...,i_{n}\}$ is finite then left approximate
identity of $A$ gives a left approximate identity for  $UP(I,A)$.
Define
 $\psi\in\Delta(UP(I,A))$ by $\psi(a)=\phi(a_{i_{n},i_{n}})$ for every
 $a=(a_{i,j})\in UP(I,A)$.  By  \cite[Theorem
3.9]{sah3} approximate biprojectivity of
 $UP(I,A)$ implies that  $UP(I,A)$ is left $\psi$-contractible, then
 the rest is similar to the proof of Theorem \ref{main}.
\end{Remark}
\begin{Proposition}
Let $A$ be a Banach algebra with a left approximate identity and
$\Delta(A)$ be a non-empty set. If $A$ is approximately
biprojective, then $\Delta(A)$ is discrete with respect to the
$w^{*}$-topology.
\end{Proposition}
\begin{proof}
Since $A$ is an approximately biprojective Banach algebra with a
left approximate identity, by \cite[Theorem 3.9]{sah3} $A$ is left
$\phi$-contractible for every $\phi\in\Delta(A)$.   Applying
\cite[Corollary 2.2]{dashti} one can see that $\Delta(A)$ is
discrete.
\end{proof}
\begin{cor}
Let $A$ be a Banach algebra with a left identity, $\phi\in\Delta(A)$
and let $I$ be a non-empty set. If $UP(I,A)$  is approximate
biprojective, then $\Delta(UP(I,A))$ is discrete  with respect to
the $w^{*}$-topology.
\end{cor}
\begin{proof}
Note that, since  $\phi\in\Delta(A)$, $\Delta(UP(I,A))$ is a
non-empty set. Existence of left identity for $A$ implies that
$UP(I,A)$ has a left approximate identity, see  \cite[Lemma
5.2]{saha}. Applying previous Proposition one can see that
$\Delta(UP(I,A))$ is discrete with respect to the $w^{*}$-topology.
\end{proof}
 Let $A$ be a Banach algebra and $\phi\in\Delta(A)$. $A$ is
$\phi$-inner amenable if there exists a bounded net $(a_{\alpha})$
in $A$ such that $$aa_{\alpha}-a_{\alpha}a\rightarrow 0,\quad
\phi(a_{\alpha})\rightarrow 1\qquad(a\in A).$$ For more information
about $\phi$-inner amenability, see \cite{jab}.
\begin{lemma}
Let $A$ be a Banach algebra and  $\phi\in\Delta(A).$ Suppose that
$A$ has an approximate identity. Then approximate biprojectivity of
$A$ implies that $A$ is $\phi$-inner amenable.
\end{lemma}
\begin{proof}
Suppose that $A$ is approximate biprojective. Using \cite[Theorem
3.9]{sah3}, existence of approximate identity implies that $A$ is
left and right $\phi$-contractible. Then there exist $m_{1}$ and
$m_{2}$ in $A$ such that
$$am_{1}=\phi(a)m_{1}(m_{2}a=\phi(a)m_{2})\quad
\phi(m_{1})=\phi(m_{2})=1\qquad(a\in A),$$ respectively. Since
$$m_{1}=\phi(m_{2})m_{1}=m_{2}m_{1}=\phi(m_{1})m_{2}=m_{2},$$ one
can see that $$am_{1}=m_{1}a=\phi(a)m_{1}\quad \phi(m_{1})=1, (a\in
A).$$ It follows that $A$ is $\phi$-inner amenable.
\end{proof}
\begin{Remark}
There exists a matrix algebra which is approximate biprojective but
it is not $\phi$-inner amenable. Then the converse of previous Lemma
is not always true.

To see this, let $A=\left(\begin{array}{cc} 0&\mathbb{C}\\
0&\mathbb{C}\\
\end{array}
\right)$ and also let $a_{0}=\left(\begin{array}{cc} 0&1\\
0&1\\
\end{array}
\right)$. Define $\rho:A\rightarrow A\otimes_{p}A$ by
$\rho(a)=a\otimes a_{0}$ for every $a\in A$. It is easy to see that
$\rho$ is a bounded $A$-bimodule morphism and $$\pi_{A}\circ
\rho(a)=a,\qquad (a\in A).$$ Then $A$ is biprojective and it follows
that $A$ is approximate biprojective.
Set $\phi(\left(\begin{array}{cc} 0&a\\
0&b\\
\end{array}
\right))=b$ for every $a,b\in \mathbb{C}$. It is easy to see that
$\phi\in\Delta(A)$. We claim that $A$ is not $\phi$-inner amenable.
We suppose conversely that $A$ is  $\phi$-inner amenable. Then there
exists a bounded net $(a_{\alpha})$ in $A$ such that
$$aa_{\alpha}-a_{\alpha}a\rightarrow 0,\quad \phi(a_{\alpha})\rightarrow 1\qquad (a\in A).$$
It is easy to see that $ab=\phi(b)a$ for every $a\in A.$ Hence we
have $$0=\lim_{\alpha} a_{0}a_{\alpha}-a_{\alpha}a_{0}=\lim
\phi(a_{\alpha})a_{0}-\phi(a_{0})a_{\alpha}=\lim a_{0}-a_{\alpha},$$
It follows that $a_{0}=\lim a_{\alpha}$. Hence for each $a\in A$, we
have $$aa_{0}=a_{0}a,\quad \phi(a_{0})=1.$$ It follows that
$a=\phi(a)a_{0}$. Thus $\dim A=1$ which is a contradiction.
\end{Remark}
\section{Examples of semigroup algebras related to the matrix algebras}
\begin{Example}
Suppose that $A$ is a Banach algebra and $I$ is a non-empty set. Put
$B=UP(I,A)$. It is obvious that $B$ with matrix multiplication can
be observed as a semigroup. Equip this semigroup with the discrete
topology and denote it with $S_{B}.$ Suppose that $A$ has a non-zero
idempotent.  We claim that $\ell^{1}(S_B)$ is not amenable, whenever
$I$ is an infinite set. Suppose conversely that $\ell^{1}(S_B)$ is
amenable. Let $e$ be an idempotent for $A$. $E_{i,i}$ for a matrix
belongs to $B$ which its $(i,i)$-th entry is $e$, otherwise is $0$.
It is easy to see that $E_{i,i}$ is an idempotent for the semigroup
$S_B$, for every $i\in I.$ So the set of idempotents of $S_{B}$ is
infinite, whenever $I$ is infinite. Thus by \cite[Theorem 2]{dun}
$\ell^{1}(S_B)$ is not amenable which is contradiction.

Suppose that $A$ is a Banach algebra with a left identity, also
suppose that $I$ is an ordered set with smallest element. We also
claim that $\ell^{1}(S_B)$ is never approximate biprojective. To see
this suppose conversely that $\ell^{1}(S_{B})$ is approximately
biprojective. We denote augmentation character on $\ell^{1}(S_B)$ by
$\phi_{S_{B}}$. It is easy to see that $\delta_{\hat{0}}\in S_{B}$
and $\phi_{S_{B}}(\delta_{\hat{0}})=1,$ where $\hat{0}$ is denoted
for the zero matrix belongs to $S_{B}.$ One can see that the center
of $S_{B}$, say $Z(S_{B})$, is non-empty, because $\hat{0}$ belongs
to $Z(S_{B})$. Using \cite[Proposition 3.1(ii)]{sah3}, one can see
that $\ell^{1}(S_{B})$ is left $\phi_{S_{B}}$-contractible. Let
$i_{0}$ be an smallest element of $I$. Define
$$J=\{(a_{i,j})\in
 S_{B}|a_{i,j}=0,\text{for all}\quad i\neq i_{0}\},$$ it is easy to
see that $J$ is an ideal of $S_{B}$, then by \cite[page 50]{dales
semi} $\ell^{1}(J)$ is a closed ideal of $\ell^{1}(S_{B})$. Since
$\phi_{S_{B}}|_{\ell^{1}(J)}$ is non-zero, $\ell^{1}(J)$ is left
$\phi_{S_{B}}$-contractible. Thus there exists  $m\in \ell^{1}(J)$
such that $am=\phi_{S_{B}}(a)m$ and $\phi_{S_{B}}(m)=1,$ for every
$a\in A.$ On the other hand since $A$ has a left  identity, then $J$
has a left identity. Thus by the same argument as in the proof of
\cite[Proposition 3.1(ii)]{sah3} we have
$$m(j)=m(e_{l}j)=\delta_{j}m(e_{l})=\phi_{S_{B}}(\delta_{j})m(e_{l})=m(e_{l})\quad (j\in J),$$
where $e_{l}$ is a left unit for $J.$ It follows that $m$ is a
constant function belongs to $\ell^{1}(J)$. Since
$\phi_{S_{B}}(m)=1,$ then $m\neq 0$ which implies that $J$ is finite
which is impossible.
\end{Example}

\end{document}